\documentclass[10pt,a4paper,leqno]{amsart}
\usepackage{amssymb}
\usepackage{graphicx}
\usepackage{pdfsync}
\usepackage{xspace}
\usepackage{fancyhdr}
\usepackage[english]{babel}
\usepackage{amsfonts}
\usepackage{latexsym}
\usepackage{graphics}
\usepackage{hyperref}
\usepackage{pifont,bbding}
\usepackage{color}
\usepackage{bm}
 \usepackage{nicefrac}
\usepackage{enumerate,amsthm,amsmath,comment,psfrag}
 \usepackage{mathrsfs} 
 \usepackage[page]{appendix}
 
\newtheorem{remark}{Remark}[section]
\newtheorem{lemma}{Lemma}[section]

\numberwithin{equation}{section}

\newcommand{\Om}{\Omega}

\newcommand{\ii}{\mathrm{i}}
\newcommand{\pa}{\bar{\partial}}

\newcommand{\Dh}{\Delta_h}
\newcommand{\E}{\mathcal{E}}

\newcommand{\Pc}{\mathcal{P}}
\newcommand{\R}{\mathcal{R}}
\newcommand{\M}{\mathcal{M}}
\newcommand{\Rc}{\texttt{Re}}
\newcommand{\Ic}{\texttt{Im}}
\newcommand{\Pb}{\mathbb{P}}
\newcommand{\Rb}{\mathbb{R}}
\newcommand{\Nb}{\mathbb{N}}
\newcommand{\Cb}{\mathbb{C}}

\newcommand{\Th}{\mathcal{T}_h}
\newcommand{\Vh}{\mathcal{V}_h}

\newcommand{\nph}{{n+\frac12}}
\newcommand{\nmh}{{n-\frac12}}
\newcommand{\bU}{ \bar{U}}
\newcommand{\dif}{\,\text{d}}
\newcommand{\intO}{\!\int_\Om\!}

\makeatletter
\newcommand*\bigcdot{\mathpalette\bigcdot@{.5}}
\newcommand*\bigcdot@[2]{\mathbin{\vcenter{\hbox{\scalebox{#2}{$\m@th#1\bullet$}}}}}
\makeatother
\begin{document}
%
%
%
%
\title[Efficient numerical approximations for NCNLS]
{
Efficient numerical approximations for\\ a non-conservative Nonlinear Schr\"odinger equation \\appearing in wind-forced ocean waves}
\date{\today}
\author{Agissilaos Athanassoulis}
\address[Agissilaos Athanassoulis]{Department of Mathematics\\
University of Dundee\\
Dundee DD1 4HN\\
Scotland, UK} \email{a.athanassoulis@dundee.ac.uk}
\author{Theodoros Katsaounis}
\address[Theodoros Katsaounis]{ Dept. of Mathematics and Applied Mathematics, Univ. of Crete, Greece \& IACM--FORTH, Heraklion, Greece
}
\email{theodoros.katsaounis@uoc.gr}
\author{Irene Kyza}
\address[Irene Kyza]{School of Mathematics and Statistics\\
University of St Andrews\\
Mathematical Institute\\ St Andrews KY16 9SS\\
Scotland, UK} \email{ika1@st-andrews.ac.uk}


\begin{abstract}
We consider a non-conservative nonlinear Schr\"odinger equation (NCNLS) with time-dependent coefficients, inspired by a water waves problem. This problem does not have mass or energy conservation, but instead  mass and energy change in time under explicit balance laws.  In this paper we extend to the particular NCNLS two numerical schemes which are known to conserve energy and mass in the discrete level for the cubic NLS. Both schemes are second oder accurate in time, and we prove that their extensions satisfy discrete versions of the mass and energy balance laws for the NCNLS. The first scheme is a relaxation scheme  that  is linearly implicit. The other scheme is a modified Delfour-Fortin-Payre scheme and it is fully implicit. Numerical results show that both schemes capture robustly the correct values of mass and energy, even in strongly non-conservative problems. We finally compare the two numerical schemes and discuss their performance. 

\vspace{1cm}
\noindent{\bf DEDICATION:} This work is dedicated to the memory of our beloved professor, colleague and friend Vassilios Dougalis.
\end{abstract}
\maketitle

%
%
\section{Introduction}
Nonlinear Schr\"odinger equations (NLS) are used in a wide range of applications as approximate or first-principles models \cite{Sulem}, including water waves \cite{Dysthe1,Dysthe2} and in particular the appearance of extreme water waves, often called rogue waves \cite{Onorato1,Onorato2,Onorato3,Janssen1,Slunyaev}. In that context, Monte Carlo simulation of stochastic sea states is often carried out \cite{Janssen1,Gramstad,Dysthe2}. More recently, it has been pointed out that richer models may be crucial in understanding rogue waves, including the presence of vorticity \cite{Curtis}, wavesystems crossing at an angle \cite{Gramstad,Athanas}, and extreme events during the growth phase of a sea state \cite{eeltnik,Pelinovsky,Brunetti,Toffoli}. This creates a need for accurate and efficient simulation for a range of nonstandard NLS-type equations.

In this work we  consider a non-conservative  nonlinear Schr\"odinger equation (NCNLS)  which arises as an envelope equation of water waves growing under wind forcing.  More specifically, we consider the following initial and boundary value problem, with either periodic or homogeneous Dirichlet boundary conditions:  we seek a \emph{wavefunction} $u : \Omega \times [0,T] \to \Cb$ such that
\begin{equation}
\label{WFNLS} \left \{
\begin{aligned}
& \ii u_t + p(t) \Delta u   + q(t) |u|^2 u + \ii r(t) u  = 0, &&\quad\mbox{$(x,t) \in {\Omega}\!\times\! (0,\ T]$,}&    \\
& u(x,0)  = u_0(x) , \qquad &&\quad\mbox{$x\in \Omega$,}& \\
& u = 0  \ \text{or}  \ u \ \text{periodic},  \qquad &&\quad\mbox{$(x,t) \in \partial\Omega \times [0, \ T]$,}&
\end{aligned}
\right.
\end{equation}%
where $T>0$ and the coefficients $p(t), \ q(t),  \ r(t)$ are  real valued and in general  depend smoothly on time. The domain $\Omega\subset\Rb^d,$ $d=1,2,3$ is bounded, convex and polygonal in the Dirichlet case, and a d-dimensional parallelepiped in the periodic case. The initial condition  is taken as $u_0 \in H^1(\Omega)$ and to satisfy the boundary conditions. 

When considering the real-world problem of growth of ocean waves during  storms,  the energy input is substantial. For example, the significant wave height (rms wave height) can grow 20 times or more. So a key requirement for numerical schemes used on the NCNLS is that they handle the energy increase in a reliable way.  To that end, we work with two schemes which are known to be conservative for the NLS, and extend each of them to the NCNLS in a way that produces exact discrete energy-balance and mass-balance laws in each case.

In particular, in Section~\ref{relax} we consider a relaxation  scheme, extending the relaxation scheme for the NLS appearing in \cite{Besse,Besse2,KK2}. A key feature is that it is linearly implicit in time, and finite elements are used for the discretization in space. With respect to the time discretization, the scheme is implicit in the Laplacian but explicit in the nonlinearity, requiring the solution of only a linear system in each time-step. In addition, it satisfies discrete analogues of the mass and energy balance laws, and it is second-order-accurate in time.

Moreover, in Section~\ref{sec:delfour}, we consider an extension of the so called Delfour-Fortin-Payre(DFP) finite element scheme \cite{Delfour, ADK} to the NCNLS. This scheme also satisfies  a discrete analogue of the energy and mass balance laws, and it is also second order accurate in time. It is however fully implicit and hence requires the solution of a nonlinear system  in every time-step. In \cite{ADK}, Newton's method was considered and analyzed for the numerical solution of the nonlinear system in each time-step.

In Section~\ref{NumEx}, we implement several numerical  experiments using the two above mentioned schemes and we show numerically that both schemes capture very robustly the correct values of mass and energy, even in strongly non-conservative cases. We compare the behaviour of the two schemes and  draw conclusions in Section \ref{sec:outlook}. 

\subsection*{Notation}
In the sequel, for $z\in\Cb$ we will denote by $\Rc(z)$ and $\Ic(z)$ its real and complex part respectively. 
We also denote by $\|\cdot\|$ and $\|\cdot\|_{L^4}$  the $L^2$- and the $L^4$-norm, respectively,   over $\Omega$. 

\textbf{\emph{Finite Element Spaces.}} For the two numerical schemes, we will use finite elements for the spatial discretization. To that end, let  $\Th$ be a conforming, shape regular partition of $\varOmega$ consisting of elements $K$ which are either simplices or $d$-dimensional cubes. 
We denote by $\Vh(\Th; \mathbb{R})$ and $\Vh(\Th; \mathbb{C})$ the real/complex finite element spaces respectively, 
\begin{equation}
\begin{aligned}
\Vh(\Th; \mathbb{R}) & :=\left\{\chi\in C(\bar{\varOmega})\cap H^1_0(\varOmega) :\forall K\in \Th, \  \chi |_K \in \Pb^\ell(K)\right\}\!, \\
\Vh(\Th; \mathbb{C}) & :=\left\{\chi_R + \ii \chi_I : \chi_R, \chi_I \in \Vh(\Th;\mathbb{R}) \right\}\!,
\end{aligned}
\end{equation}
where $\Pb^\ell(K)$ denotes the space of polynomials on the element $K$ of total degree $\ell$ if $K$ is a simplex or of degree $\ell$ in each variable if $K$ is a $d$-dimensional cube.  We assume that at each time step $n$ the mesh $\Th$ does not change. 

To define the schemes in Sections~\ref{relax} and ~\ref{sec:delfour}, we  introduce two operators; the $L^2$-projection operator $\Pc_h : L^2(\varOmega) \to \Vh(\mathbb{C})$ and the discrete Laplacian operator $-\Dh : H_0^1(\varOmega) \to \Vh(\mathbb{C})$, which are defined implicitly as the solution  of the following variational problems
\begin{align}
& v\mapsto \Pc_h v,  &&\hspace{-15mm} \langle \Pc_h v, \chi \rangle = \langle v , \chi \rangle, &\hspace{-15mm}\forall \chi \in \Vh(\mathbb{R}), \label{L2pr} \\
& v \mapsto -\Dh v, &&\hspace{-15mm} \langle -\Dh v, \chi \rangle = \langle \nabla v , \nabla \chi \rangle, &\hspace{-15mm} \forall \chi \in \Vh^n(\mathbb{R}), \label{DLapl}
\end{align}
where $\langle \cdot, \cdot \rangle$ denotes the $L^2$-inner product over $\varOmega$. Note that although the $L^2$-projection/discrete Laplacian may be complex in the above definitions, the test functions always lie in the real finite element space $\Vh(\mathbb{R})$. 
%
%
%
\section{Continuous Mass and Energy Balance}
The standard NLS  equation, i.e. taking $r(t)\equiv 0$ and $p(t), q(t)$ independent of time in \eqref{WFNLS},  satisfies a series of \emph{conservation laws}, with \emph{mass} and  \emph{energy} being the most fundamental and of great physical importance in many applications. However, when the coefficients $p(t), q(t), r(t)$ vary with time, \eqref{WFNLS} satisfies more general mass and energy \emph{balance laws} instead of conservation. 
We define the \emph{mass} $\M(t)$,  the \emph{kinetic} $\E_\kappa(t)$ and \emph{potential} $\E_p(t)$ energies respectively  
\begin{equation}
\label{massenergy}
 \M(t) :=\|u(t)\|^2 ,\quad \E_\kappa(t) : = \|\nabla u\|^2,  \quad \E_p(t):=\|u\|_{L^4}^4.
\end{equation}
Then one can show the following
\begin{lemma}[Continuous Mass and Energy balance] \label{lm:1}
If $u$ is a solution of \eqref{WFNLS} then for $0\le t \le T$ we have 
\begin{align}
 \frac{d}{dt}\M(t) & = -2r(t)\M(t) 
 &&\text{Balance of Mass},  \label{mbl} \\
\frac12 p(t)\frac{\dif }{\dif t} \E_\kappa(t)  + \frac14 q(t) \frac{\dif }{\dif t} \E_p(t) & = r(t)\Big(p(t)\E_\kappa(t) - q(t)\E_p(t)\Big) &&\text{Balance of Energy}. \label{ebl}
\end{align}
\end{lemma}
\begin{proof}
We derive first the balance of mass. We multiply the Schr{\"o}dinger equation \eqref{WFNLS} by $\bar{u}$ and integrate over $\Omega$ yielding
\begin{equation*}
\ii \!\intO \!\!u_t \bar{u}\dif x - p(t)\|\nabla u\|^2  + q(t) \|u\|_{L^4}^4 + \ii r(t) \|u\|^2= 0,
\end{equation*}
Since $p, q, r$ are real valued, taking imaginary parts, we obtain 
\begin{equation*}
\frac12 \frac{\dif }{\dif t} \|u\|^2 = -r(t) \|u\|^2 
\end{equation*}
which yields immediately \eqref{mbl}.

For the energy balance we multiply the Schr{\"o}dinger equation \eqref{WFNLS} by $\bar{u}_t$ and integrate over $\Omega$ to get,
\begin{equation*}
\ii \!\intO \!\! u_t \bar{u}_t\dif x + p(t)\!\intO \!\!\Delta u\bar{u}_t \dif x  + q(t)\!\!\intO \!\!|u|^2 u \bar{u}_t \dif x + \ii r(t)\!\! \intO \!\!u \bar{u}_t \dif x= 0.
\end{equation*}
Integrating by parts the second term and taking real parts we obtain, 
\begin{equation}
\label{ebla}
\frac12 p(t) \frac{\dif }{\dif t} \|\nabla u \|^2  + \frac14 q(t)\frac{\dif }{\dif t} \|u\|_{L^4}^4 + \Rc\left(\ii r(t)\!\!\intO \!\!u \bar{u}_t \dif x \right)= 0 .
\end{equation}
For the third term in \eqref{ebla}, using \eqref{WFNLS} we have 
\begin{align*}
& \Rc\left(\ii r(t)\!\!\intO \!\!u \bar{u}_t \dif x \right) = \Rc\left(r(t)\!\!\intO \!\!u (-\overline{\ii u}_t )\!\dif x \right) = 
 \Rc\left(r(t)\!\!\intO \!\!u \left( p(t) \Delta \bar{u}  + q(t) |u|^2 \bar{u} + \ii r(t) \bar{u}  \right)\!\!\dif x \right) = \\
  & \Rc\left(r(t)\!\!\intO \!\!u \left( p(t) \Delta \bar{u}  + q(t) |u|^2 \bar{u}  \right)\!\!\dif x \right) = r(t)\left(-p(t)\|\nabla u\|^2 + q(t)\|u\|_{L^4}^4\right) , 
\end{align*}
and combined with \eqref{ebla} yields \eqref{ebl}.
\end{proof}
\begin{remark}\upshape Solving
equation \eqref{mbl} it can be seen that
\begin{equation}\label{mb1b}
\mathcal{M}(t) = \mathcal{M}(0) e^{\textstyle{-2\int_0^t r(s) \dif s}}. 
\end{equation}
In particular, depending on the sign of $r(t)$, mass can experience either exponential growth or decay.  \end{remark}
\begin{remark}\upshape
In case $p, q, r$ are independent of time, then \eqref{mbl} and \eqref{ebl} simplify
\begin{align} \label{mc}
 \M(t) & = \M(0)e^{\textstyle{-2r_0 t}} \quad 0\le t\le T, \\
 \frac{\dif }{\dif t} \left(\frac12 p_0 \E_\kappa(t)  + \frac14 q_0 \E_p(t) \right) & = r_0 \left(p_0\E_\kappa(t) - q_0 \E_p(t)\right), \quad 0\le t\le T.
 \label{ec}
\end{align}
If, furthermore,  $r\equiv 0$ then we recover the standard conservation of mass and energy for the NLS. 
\end{remark}
%
%
%
\section{A relaxation numerical scheme and discrete balance laws  }\label{relax}
We introduce now a numerical scheme for approximating solutions of \eqref{WFNLS}. Our  method is based on the relaxation approach introduced by Besse \cite{Besse} for the classical NLS. A similar approach was also considered for the  Schr\"odinger-Poisson system in \cite{AKK}. 
First we rewrite \eqref{WFNLS} as a coupled system of two equations by introducing an auxiliary variable $\phi$ for the nonlinearity, and we consider the following enlarged system 
 \begin{equation}
\label{RWFNLS} \left \{
\begin{aligned}
&  u_t -\ii p(t) \Delta u   -\ii q(t) \phi u +  r(t) u = 0, &&\quad\mbox{$(x,t) \in {\Omega}\!\times\! (0,\ T]$,}&    \\
& \phi = |u|^2, &&\quad\mbox{$(x,t) \in {\Omega}\!\times\! (0,\ T]$,}&    \\
& u(x,0)  = u_0(x) , \qquad &&\quad\mbox{$x\in \Omega$,}& \\
& u = 0  \ \text{or}  \ u \ \text{periodic},  \qquad &&\quad\mbox{$(x,t) \in \partial\Omega \times [0, \ T]$,}&
\end{aligned}
\right.
\end{equation}%
which is equivalent to \eqref{WFNLS}. The numerical scheme that we introduce next is based on this enlarged form of \eqref{WFNLS}. To simplify the presentation we assume in the sequel homogeneous Dirichlet boundary conditions, while periodic boundary condition can be easily incorporated with slight modifications. We present first the time discrete scheme and then we proceed to the fully-discrete scheme. 

\subsection{Time discrete scheme}
We introduce a sequence of nodes $0=:t_0 < \dots < t_n < \dots < t_N := T $ of $[0, T]$ and variable time steps $k_n = t_{n+1} - t_n$. Then the relaxation scheme for \eqref{WFNLS} based on \eqref{RWFNLS} is defined as follows: we seek approximations $U^n \in H_0^1(\Omega)$ to $u(t_n) \in H_0^1(\Omega), n=1,\dots, N$ such that 
\begin{equation}
\label{TDrelax} \left \{
\begin{aligned}
& \frac{k_{n-1}}{k_n+k_{n-1}}\Phi^{\nph} + \frac{k_{n}}{k_n+k_{n-1}}\Phi^{n-\frac12}=  |U^n|^{2},\quad 0\le n\le N-1,    \\
&\pa U^n-\ii p_{\nph}\Delta U^{\nph}  - \ii q_{\nph} \Phi^{\nph}U^{\nph} + r_{\nph}U^{\nph} = 0 ,\qquad 0\le n\le N-1,
\end{aligned}
\right.
\end{equation}
where we have used the notation 
\begin{equation}
\label{notation}
t_\nph = \frac{t_{n+1}+t_n}{2}, \  \pa U^n:=\frac{U^{n+1}-U^n}{k_n}\ \text{ and }\ U^{\nph}:=\frac{U^{n+1}+U^n}2 , 
\end{equation}
with  $f_{\nph} = f(t_{\nph})$ for $f = p, q, r$.  At $n=0$, a straightforward choice to initialize the system would be  $U^0=u_0$, $k_{-1}=k_0$, $\Phi^{-\frac12}=|u_0|^{2}$,  \cite{Besse,Besse2}. However this simple choice has been known to create  computational issues for other equations \cite{Zouraris,AKK} and the same seems to apply here. Thus, in Section \ref{sec:initi} we introduce a modified initialization which addresses this issue. 

\subsection{Fully discrete scheme}
We introduce now a fully discrete scheme for \eqref{WFNLS} based on \eqref{RWFNLS}. The time discrete scheme \eqref{TDrelax} can be combined with various methods for spatial discretization, including finite differences \cite{Besse,Zouraris1} and spectral methods. Motivated by our previous works for  Schr\"odinger-type models, \cite{KK}, \cite{KK2}, \cite{AKK} we choose finite elements for space discretization.

We  introduce now the fully-discrete  relaxation scheme for \eqref{WFNLS} based on \eqref{RWFNLS} which is given as follows: we seek approximations $U_h^{n}\in \Vh(\mathbb{C})$ to $u(\cdot, t_{n}) \in H^1_0(\varOmega)$, $1\le n \le N$, such that
\begin{equation}
\label{FDrelax} \left \{
\begin{aligned}
& \frac{k_{n-1}}{k_n+k_{n-1}}\Phi_h^{\nph} + \frac{k_{n}}{k_n+k_{n-1}}\Phi_h^{n-\frac12}= \Pc_h\left( |U_h^n|^{2}\right),\quad 0\le n\le N-1,    \\
&\pa U_h^n-\ii p_{\nph}\Dh U_h^{\nph}  - \ii q_{\nph} \Pc_h\left(\Phi_h^{\nph}U^{\nph} \right)+ r_{\nph}U_h^{\nph} = 0 ,\qquad 0\le n\le N-1,
\end{aligned}
\right.
\end{equation}
with $U_h^0=\Pc_h u_0$, $k_{-1}=k_0$ and, for the straightforward initialisation, $\Phi_h^{-\nicefrac{1}{2}}=\Pc_h(|u_0|^2).$ Recall that $\Pc_h$ and $\Dh$ denote the $L^2$-projection operator \eqref{L2pr} and the discrete Laplacian operator \eqref{DLapl}, respectively.  In the case of constant timestep $k_n=k$ we get 
\begin{equation}
\label{FDrelax1} \left \{
\begin{aligned}
&\Phi_h^{\nph} = 2\Pc_h\left( |U_h^n|^{2}\right) - \Phi_h^{\nmh},\quad 0\le n\le N-1,    \\
&\pa U_h^n-\ii p_{\nph}\Dh U_h^{\nph}  - \ii q_{\nph} \Pc_h\left(\Phi_h^{\nph}U^{\nph} \right)+ r_{\nph}U_h^{\nph} = 0 ,\qquad 0\le n\le N-1,
\end{aligned}
\right.
\end{equation}
%
%
\subsection{Discrete mass and energy balance laws}
We define the \emph{discrete mass}  $\M_h^n$ as the discrete counterpart of $\M$ in \eqref{massenergy}, namely $\M_h^n = \|U_h^n\|^2$. The discrete mass satisfies a discrete analogue of the continuous mass balance \eqref{mbl}, 
\begin{lemma}[Local Discrete Mass balance]
The solution of the fully discrete relaxation scheme \eqref{FDrelax} satisfies 
\begin{equation}
\label{dmbl}
\M_h^{n+1} = \M_h^n - 2 k_n r_{\nph} \|U_h^{\nph}\|^2, \quad \mbox{ where } \quad \M_h^n = \|U_h^n\|^2.
\end{equation}
\end{lemma}
\begin{proof}
We multiply \eqref{FDrelax}(b) by $\bar{U}_h^{\nph}$ and we integrate 
\begin{equation*}
\intO\!\! \bar{U}_h^{\nph} \pa U_h^n \dif x +\ii p_{\nph} \|\nabla U_h^{\nph}\|^2  - \ii q_{\nph}\intO \Phi_h^{\nph} | U_h^{\nph}|^2 \dif x+ r_{\nph} \|U_h^{\nph}\|^2 = 0 
\end{equation*}
taking real parts we get 
\begin{equation*}
\Rc\left( \intO \bar{U}_h^{\nph} \pa U_h^n \dif x \right) + r_{\nph} \|U_h^{\nph}\|^2  = 0 
\end{equation*}
Also 
\begin{equation*}
\bar{U}_h^{\nph} \pa U_h^n =  \frac1{2k_n}\left( \bar{U}_h^{n+1} +\bar{U}_h^n\right) \left(U_h^{n+1} - U_h^n\right) = \frac1{2k_n}\left( |U_h^{n+1}|^2 - |U_h^n|^2 - (\bar{U}_h^{n+1}U^n - \bar{U}_h^{n} U_h^{n+1}) \right) 
\end{equation*}
Since $\Rc\left( \bar{U}_h^{n+1}U_h^n - \bar{U}_h^{n} U_h^{n+1}\right) = 0$, the last two relations yield \eqref{dmbl}.  
\end{proof}
\begin{remark}\upshape
Relation \eqref{dmbl} describes a local discrete mass balance in $[t_n, t_{\nph}]$. If $r(t)\equiv 0$ then \eqref{dmbl} reduces to the standard global discrete mass conservation for the NLS  \eqref{mc}: $\M_h^n = \M_h^0$ .
\end{remark}
We show next a discrete energy balance for the fully discrete relaxation scheme. To do so, we need to define the \emph{discrete kinetic} energy $\E_{\kappa,h}^n$
\begin{equation}
\label{ddmebl0}
\E_{\kappa,h}^n = \|\nabla U_h^n\|^2, 
\end{equation} 
 and \emph{discrete potential energy} $\E_{p,h}^n.$ In particular, $\E_{p,h}^n$ is obtained through a non-trivial discretization of the potential energy, namely
\begin{equation}
\label{ddmebl}
\E_{p_1,h}^n = \intO  |\Phi_h^\nmh|^2\!\dif x, \quad \E_{p_2,h}^n = \intO \Phi_h^\nmh | U_h^n|^2 \!\dif x, \quad \E_{p,h}^n :=   2  \E_{p_2,h}^n - \E_{p_1,h}^n 
\end{equation} 
The role of the \emph{discrete potential energy} $ \E_{p,h}^n$ will be carried by $ \E_{p,h}^n =  2  \E_{p_2,h}^n - \E_{p_1,h}^n.$
\begin{lemma}[Local Discrete Energy balance] The solution of the fully discrete relaxation scheme \eqref{FDrelax} satisfies 
\begin{equation}
\begin{aligned}
\label{debl}
 \frac12 p_{\nph}\pa\left(\E_{\kappa,h}^n\right) + \frac14 q_{\nph} \pa\left( \E_{p,h}^n \right) & = r_{\nph}\left(p_{\nph} \|\nabla U_h^{\nph}\|^2 - q_{\nph}\!\!\intO\!\!\Phi_h^{\nph} |U_h^{\nph}|^2\! \dif x\right)
\end{aligned}
\end{equation}
\end{lemma}
\begin{proof}
We multiply \eqref{FDrelax}(b) by $\pa\bar{U}^n$ and we integrate over $\Omega:$ 
\begin{align*}
&\intO \pa\bU_h^n\pa U_h^n \dif x - \ii p_\nph \intO \Delta U_h^\nph \pa\bU_h^n \dif x - \ii q_\nph \intO \Phi_h^\nph U_h^\nph \pa\bU_h^n \dif x + r_\nph\intO U_h^\nph \pa\bU_h^n \dif x = 0 \\
&\|\pa U_h^n\|^2 + \ii p_\nph\intO\nabla U_h^\nph\cdot \nabla \pa \bU_h^n \dif x  - \ii q_\nph \intO\Phi_h^\nph U_h^\nph \pa\bU_h^n \dif x + r_\nph\intO U_h^\nph\pa\bU_h^n \dif x = 0 
\end{align*}
Taking imaginary parts we obtain 
\begin{equation}
\label{debla}
\frac{p_\nph}{2}\pa\left( \|\nabla U_h^n\|^2 \right) - \Ic\left( \ii q_\nph \intO\Phi_h^\nph U_h^\nph \pa\bU_h^n \dif x\right) + \Ic\left( r_\nph \intO U_h^\nph \pa\bU_h^n \dif x \right)  = 0 
\end{equation}
Using \eqref{TDrelax}(b), for the third term in \eqref{debla}, we have 
\begin{equation}
\begin{aligned}
\label{deblaa}
 \intO U_h^\nph \pa\bU_h^n \dif x & = \intO U_h^\nph\left( \ii p_\nph\Delta\bU_h^\nph + \ii q_\nph\Phi_h^\nph\bU_h^\nph - r_\nph\bU_h^\nph\right)\dif x  \\
& = -\ii p_\nph \|\nabla U_h^\nph\|^2 + \ii q_\nph \intO\Phi_h^\nph |U_h^\nph|^2 \dif x - r_\nph\|U_h^\nph\|^2  \implies \\
 \Ic\left( r_\nph \intO U_h^\nph \pa\bU_h^n \dif x \right) & = r_\nph\left( - p_\nph \|\nabla U_h^\nph\|^2  + q_\nph \intO\Phi_h^\nph |U_h^\nph|^2 \dif x    \right) 
\end{aligned}
\end{equation}
Similarly, for the second term in \eqref{debla} we have 
\begin{equation}
\label{deblab}
\begin{aligned}
 \Ic\left( \ii q_\nph \right. & \left. \intO\Phi_h^\nph U_h^\nph \pa\bU_h^n \dif x\right)  = \frac{q_\nph}{2k_n} \intO \Phi_h^\nph\left(|U_h^{n+1}|^2 - |U_h^n|^2\right)\dif x = \\
& \frac{q_\nph}{2k_n}\intO\left( \Phi_h^\nph  |U_h^{n+1}|^2 - \Phi_h^\nmh|U_h^n|^2 + \Phi_h^\nmh|U_h^n|^2 - \Phi_h^\nph |U_h^n|^2\right) \dif x = \\
&  \frac{q_\nph}{2}\intO\pa\left(\Phi_h^\nmh|U_h^n|^2\right) \dif x + \frac{q_\nph}{2k_n}\intO |U_h^n|^2\left(\Phi_h^\nmh - \Phi_h^\nph\right)\dif x  \overset{\eqref{TDrelax}(a)}{=} \\
&  \frac{q_\nph}{2}\intO\pa\left(\Phi_h^\nmh|U_h^n|^2\right) \dif x + \frac{q_\nph}{4k_n}\intO \left(\Phi_h^\nph+\Phi_h^\nmh\right)\left(\Phi_h^\nmh - \Phi_h^\nph\right)\dif x = \\
&  \frac{q_\nph}{2}\intO\pa\left(\Phi_h^\nmh|U_h^n|^2\right) \dif x - \frac{q_\nph}{4}\intO \pa\left(|\Phi_h^\nmh|^2 \right)\dif x 
\end{aligned}
\end{equation}
Combining  \eqref{debla}, \eqref{deblaa}, \eqref{deblab} \eqref{ddmebl0} and \eqref{ddmebl}, we obtain \eqref{debl}. 
\end{proof}
\begin{remark}[\emph{Relationship between discrete and continuous balance laws}]\label{rm:corr1}
\upshape
The discrete mass balance law, equation \eqref{dmbl}, is equivalent to
\[
\pa \mathcal{M}^n_h = -2 r_{n+\frac12}\|U^{n+\frac12}_h\|^2
\]
which is a direct discretization of the continuous mass balance law, equation \eqref{mbl},
\[
 \frac{d}{dt}\M(t)  = -2r(t)\M(t).
\] 
Moreover, the continuous energy balance law, equation \eqref{ebl},  can be understood with 
\[
\E_\kappa(t) : = \|\nabla u\|^2,  \quad \E_p(t):=\|u\|_{L^4}^4= \langle2|u|^2-\phi,\phi\rangle
\]
in the context of the augmented system \eqref{RWFNLS}.
Thus the discrete energy balance law, equation \eqref{debl},
\[
 \frac12 p_{\nph}\pa\left(\E_{\kappa,h}^n\right) + \frac14 q_{\nph} \pa\left( \E_{p,h}^n \right)  = r_{\nph}\left(p_{\nph} \|\nabla U_h^{\nph}\|^2 - q_{\nph}\!\!\intO\!\!\Phi_h^{\nph} |U_h^{\nph}|^2\! \dif x\right)
\]
is a direct discretization of the continuous energy balance law, equation \eqref{ebl},
\[
\frac12 p(t)\frac{\dif }{\dif t} \E_\kappa(t)  + \frac14 q(t) \frac{\dif }{\dif t} \E_p(t) = r(t)\Big(p(t)\E_\kappa(t) - q(t)\E_p(t)\Big).
\]

Unlike the conservative setting,  discrete mass and energy balance laws do not automatically guarantee that the mass and energy are exactly correct at all times, only that their variation mimicks in a precise way that of the continuous problem. The same is true for the fully implicit scheme \eqref{Delfour} which is examined in the next Section. Thus the  ultimate validation of the discrete mass and energy laws is how close the mass and energy of the numerical solution are to the exact value, cf. Tables \ref{mberr}-\ref{meeoc} in Section \ref{NumEx}. \end{remark}
%
%
%
\section{A fully implicit scheme and discrete balance laws}\label{sec:delfour}
%
We also consider a fully discrete scheme based on the one  proposed by Delfour-Fortin-Payre, in \cite{Delfour} for the classical cubic NLS. Its main ingredient  is how the nonlinearity is treated, leading to  exact conservation, at the discrete level, of mass and energy for classical cubic NLS with constant coefficients.  However, the scheme is fully nonlinear and a linearization process is required to obtain the approximate solution. A variation of Newton's method  for this scheme was fully analyzed in \cite{ADK}. 

We present now the modified DFP scheme in the context of our model \eqref{WFNLS}: we seek approximations $U_h^{n}\in \Vh(\mathbb{C})$ to $u(\cdot, t_{n}) \in H^1_0(\varOmega)$, $1\le n \le N$, such that
\begin{equation}
\label{Delfour}\left\{
\begin{aligned}
& \pa U_h^n - \ii p_\nph \Dh U_h^\nph - \ii \frac{q_\nph}{2} \Pc_h\left( \left(  |U_h^{n+1}|^2 + |U_h^n|^2 \right) U_h^\nph   \right) + \ii r_\nph U_h^\nph = 0 ,\\
& U_h^0 = \Pc_h(u_0) .
\end{aligned}
\right.
\end{equation}
We proceed by examining the conservation properties of \eqref{Delfour}. Concerning the mass balance,  scheme \eqref{Delfour} satisfies the exact same balance relation \eqref{dmbl} as the relaxation scheme \eqref{FDrelax}, and the proof is straightforward, thus we omit its presentation. However, the corresponding energy balance  that  \eqref{Delfour} satisfies differs from \eqref{debl} and we present it next. 
\begin{lemma}[Local discrete energy balance]
The solution of the fully discrete scheme \eqref{Delfour} satisfies 
\begin{multline}
\label{deblDel} 
\frac12 p_\nph \pa\|\nabla U_h^n\|^2 + \frac14 q_\nph \pa\|U_h^n\|_{L^4}^4= \\ = r_\nph\left( p_\nph \|\nabla U_h^\nph\|^2 -  \frac{q_\nph}{2} \intO\left( |U_h^{n+1}|^2 + |U_h^n|^2 \right) |U_h^\nph|^2 \!\dif x  \right).
\end{multline}
\end{lemma}
\begin{proof}
We multiply \eqref{Delfour} by $\pa\bU^n$ and we integrate to get 
\begin{multline*}
\|\pa U_h^n\|^2 + \ii p_\nph\intO\nabla U_h^\nph\cdot \nabla \pa \bU_h^n \dif x  \\ - \ii \frac{q_\nph}{2} \intO\left( |U_h^{n+1}|^2 + |U_h^n|^2\right) U_h^\nph \pa\bU_h^n \dif x + r_\nph\intO U_h^\nph\pa\bU_h^n \dif x = 0 
\end{multline*}
Taking imaginary parts we obtain 
\begin{multline*}
\frac{p_\nph}{2}\pa\left( \|\nabla U_h^n\|^2 \right) - \Ic\left( \ii \frac{q_\nph}{2} \intO\left( |U_h^{n+1}|^2  + |U_h^n|^2\right) U_h^\nph \pa\bU_h^n \dif x\right) \\ + \Ic\left( r_\nph \intO U_h^\nph \pa\bU_h^n \dif x \right)  = 0 
\end{multline*}
For the second term we have 
\begin{multline*}
\Ic\left( \ii \frac{q_\nph}{2} \intO\left( |U_h^{n+1}|^2  + |U_h^n|^2\right) U_h^\nph \pa\bU_h^n \dif x\right)= \\ = \frac{q_\nph}{2} \frac{1}{2k_n}  \intO\left( |U_h^{n+1}|^2  + |U_h^n|^2\right)\left( |U_h^{n+1}|^2  - |U_h^n|^2\right) \! \dif x
\end{multline*}
As before, see \eqref{deblaa}, for the third term we have 
\begin{equation*}
\Ic\left( r_\nph \intO U_h^\nph \pa\bU_h^n \dif x \right)  = r_\nph\left( - p_\nph \|\nabla U_h^\nph\|^2  + \frac{q_\nph}{2} \intO \left( |U_h^{n+1}|^2  + |U_h^n|^2\right) |U_h^\nph|^2 \dif x    \right) 
\end{equation*}
and the result follows by combining the last three relations. 
\end{proof}

\begin{remark}\upshape
In the case $r(t)\equiv 0$ and $p, q$ are constants then from \eqref{deblDel} we recover the standard conservation of energy for the cubic NLS, namely 
\begin{equation*}
\frac12 p_0 \|\nabla U_h^n\|^2  + \frac14 q_0 \|U_h^n\|_{L^4}^4 = \frac12 p_0 \|\nabla U_h^0\|^2  + \frac14 q_0 \|U_h^0\|_{L^4}^4.
\end{equation*} 
In the context of the NCNLS,
 the extended  DFP scheme \eqref{Delfour} offers a more straightforward discretization of the potential energy appearing in the left hand side of equation \eqref{deblDel} compared to the relaxation scheme, namely $\|U^n_h\|_{L^4}^4$ (cf. equations \eqref{ddmebl} and \eqref{debl} for comparison). This, however, comes at the cost of scheme \eqref{Delfour} being fully nonlinear, in contrast to the linearly implicit scheme \eqref{FDrelax}. Observe moreover that a nontrivial discretization of the potential energy appears on the right hand side of \eqref{deblDel} too, namely $\frac{1}{2} \intO\left( |U_h^{n+1}|^2 + |U_h^n|^2 \right) |U_h^\nph|^2 \!\dif x.$
\end{remark}
%
%
\section{Numerical Experiments}\label{NumEx}
We perform a series of numerical experiments to validate the numerical methods and study their behaviour. In the process, we discuss some details about the implementation of the schemes.  The numerical results reported in this section are one dimensional ($d=1$) and the numerical scheme is implemented with  in house  \texttt{C}-codes using double precision arithmetic.

\subsection{Implementation details for the relaxation scheme}\label{sec:initi}
In this section we present some technical details concerning the efficient implementation of scheme \eqref{FDrelax1} and its appropriate initialization. We rewrite the numerical scheme \eqref{FDrelax1} as a Runge-Kutta method : we first update $\Phi^{\nph}$, we then solve for the intermediate stage $U_h^{\nph}$ and finally we update $U_h^{n+1}$ for $0\le n\le N-1$,
\begin{equation}
\label{FDrelax2} \left \{
\begin{aligned}
&\Phi_h^{\nph} = 2\Pc_h\left( |U_h^n|^{2}\right) - \Phi_h^{\nmh},   \\
& \left(1+\frac{k}{2} r_{\nph}\right)U_h^{\nph} -\ii\frac{k}{2} p_{\nph}\Dh U_h^{\nph}  - \ii \frac{k}{2} q_{\nph} \Pc_h\left(\Phi_h^{\nph}U^{\nph} \right) = k U_h^{n}, \\
& U_h^{n+1} = 2 U_h^{\nph} - U_h^n .
\end{aligned}
\right.
\end{equation}
A simple choice for initializing \eqref{FDrelax2} is to take $U_h^0 = \Pc_h u_0$ and $\Phi_h^{-\frac12} = \Pc_h(|u_0|^2)$. For these initial values, it is observed numerically that  $U_h^n$ is a second order approximation in time to $u(t_n)$, but $\Phi_h^\nph$ is only first order approximation to $|u(t_\nph)|^2$. Consequently, this affects scheme's behaviour towards mass and energy balance laws. The same phenomenon was observed also in \cite{Zouraris,AKK,Zouraris2} for similar schemes applied to different equations. In those same references an alternative approach to initialization was proposed to remedy the situation, which we describe below. 
\begin{align*}
& \text{Initialize:} \quad U_h^0 = \Pc_h u_0 \\
& \text{Use the scheme \eqref{FDrelax2} with naive initialization and half timestep to compute:} \quad U_h^{\frac12} \\
& \text{Set :}\quad \Phi_h^{\frac12} = \Pc_h(|U_h^{\frac12}|^2) \\
& \text{Use this $\Phi_h^{\frac12}$ and the $2^{{nd}}$ equation in scheme \eqref{FDrelax2} to compute :} \quad U_h^1 \\
& \text{Continue with scheme \eqref{FDrelax2} for :} \quad n=1,\dots, N-1
\end{align*}
This initialization process is observed numerically to lead to second order approximation in time for both $U_h^n, \ \Phi_h^\nph,$ cf. Tables \ref{eocs}, \ref{eoct} and \ref{meeoc}. In the remainder of this Section, when we refer to scheme \eqref{FDrelax2} we will always mean that the improved initialization discussed above is used.

\subsection{Implementation of the DFP scheme}

The modified DFP  scheme \eqref{Delfour} is implicit, i.e. it requires the solution of a nonlinear equation for $U^{n+1}$ at each timestep. This can be done with a Newton method; a detailed strategy is presented  in \cite{ADK} and convergence is proven under precise assumptions. To describe the main idea, let us denote by $U^{n+1}_{m}$  the sequence of approximate solutions generated by Newton iteration for $U^{n+1}.$ This sequence is initialized with $U^{n+1}_{(0)}=U^n.$ Now in order to compute the first Newton iteration step, $U^{n+1}_{(1)},$ one still needs to solve a nonlinear problem. To that end, a second ``inner'' iteration is required; a fixed point iteration is setup for that. It is found in \cite{ADK} that one step of the Newton iteration with 4 steps of the inner iteration suffices.

\subsection{Manufacturing exact solutions for the NCNLS and validation}
Let $w(x,t)$ be a solution of the classical NLS equation, i.e. a solution of \eqref{WFNLS} with $p=p_0,$ $q=\theta_0,$ $r=0,$
  \begin{equation*}
 \ii w_t + p_0 \Delta w   + \theta_0 |w|^2 w  = 0.
\end{equation*}
Then, given any smooth function $r(t),$ it follows that the function
$$
u(x,t) := w(x,t) \exp{\left(-\int_0^t r(s) \dif s \right)}
$$
satisfies
$$
 \ii u_t + p_0 \Delta u   + q(t) |u|^2 u + \ii r(t) u  = 0, \qquad q(t)=\theta_0\exp{\left(2\int_0^t r(s) \dif s \right) }.
$$

Thus starting from a standard soliton solution of the classical NLS we can  generate the solution
\begin{equation}
\label{exact}
u(x,t) = \ii \exp\left(\ii\left(2\omega x + (1-4\omega^2)t \right)  \right)\exp\left( \int_0^t r(s) \dif s\right) \text{sech}\left(x-4\omega t\right)
\end{equation}
of \eqref{WFNLS} for any smooth function $r(t).$ (This requires that the effective support of the initial soliton is well contained within $\Omega;$ then periodic BCs can be used with no problem. Formula \eqref{exact} applies until the effective support of the soliton reaches the boundary. For periodic BCs, a periodized soliton can still be used for long times.)

All computations reported in Tables \ref{eocs}, \ref{eoct}, \ref{Deocs} and \ref{Deoct},  were performed for $\Omega \times [0,T] =  [a,b] \times [0,T] = [-30,30] \times [0,1]$ with  $\omega=0.3$ and taking $r(t)=\sin\left(\frac{2\pi}{T}t\right)$. For $M,  N \in\Nb$ we consider uniform partitions of $\Omega \times [0,T]$ according to $h = \frac{b-a}{M}, \ k = \frac{T}{N}$ with $x_i = a + i h, \ i=0,1,\dots, M$, $t_n = n k, \ n=0,1,\dots, N$. For a given pair $(M,N)$ we implement the method with corresponding mesh sizes; in order to investigate the EOC, sequences of runs with different mesh sizes $(h_\mu, \ k_\nu)$ are used.  The errors are computed in the $L^{\infty}(L^2)$-norm and formally we expect that 
\begin{align*}
& E(u; h,k) := \max_{0\le n\le N} \|u(\cdot, t_n) - U_h^n\| = O(h^{\ell+1}+k^2), \\ 
& E(\phi; h,k) := \max_{0\le n\le N} \| |u(\cdot,t_{\nph})|^2 - \Phi_h^{\nph}\| = O(h^{\ell+1}+k^2),
\end{align*}
where $\ell$ is the order of polynomials used in the finite element space. The spatial ($\R_s$)  and temporal ($\R_t$) EOC (convergence rates) for either $u$ or $\phi$ are then computed as 
\begin{equation*}
\R_s = \frac{\log\left(E( h_1,k)\right) - \log\left(E(h_2,k)\right)}{\log(h_1) -\log(h_2)}, \quad 
\R_t = \frac{\log\left(E(h,k_1)\right) - \log\left(E(h,k_2)\right)}{\log(k_1) -\log(k_2)},
\end{equation*}
respectively, for two successive runs with mesh sizes $h_1,  h_2$ and fixed $k$, or two time steps $k_1,  k_2$ and fixed  mesh size $h$.  To compute $\R_s$ we consider a very fine  timestep $k=10^{-5}$, thus the temporal component of the error is negligible, and perform a series of runs with different $h_\mu$.  In Table \ref{eocs}, the spatial experimental orders of convergence of the modified relaxation scheme are displayed for $\ell =1, 2$. The optimal rate of convergence is observed in both cases,  thus validating the claimed spatial accuracy of the numerical scheme \eqref{FDrelax2}.
\begin{table}[htp]
\caption{Spatial experimental rates of  convergence $\R_s$ for $u, \phi$ for scheme \eqref{FDrelax2}.}\vspace{-2mm}
\begin{center}
\begin{tabular}{||c|cc|cc||cc|cc||} \hline
&  \multicolumn{4}{c||}{$\ell=1$} & \multicolumn{4}{c||}{$\ell=2$}  \\ \hline 
$h_\mu$ & $E(u;h_\mu,k)$ & $\R_s$  &  $E(\phi;h_\mu,k)$ & $\R_s$ & $E(u;h_\mu,k)$ &  $\R_s$ &  $E(\phi;h_\mu, k)$ & $\R_s$\\ \hline 
6.00e-01	&	1.7256e-01	&	-		&	1.4495e-01	&	-		&		1.3066e-02	&	-		&	1.6119e-02	&	-		\\ \hline
3.00e-01	&	4.8429e-02	&	1.833	&	4.3372e-02	&	1.741	&		7.0728e-04	&	4.207	&	8.6762e-04	&	4.216	\\ \hline
1.50e-01	&	1.2493e-02	&	1.955	&	1.1419e-02	&	1.925	&		5.9636e-05	&	3.568	&	7.7301e-05	&	3.489	\\ \hline
1.20e-01	&	8.0261e-03	&	1.983	&	7.3556e-03	&	1.971	&		2.8734e-05	&	3.272	&	3.7697e-05	&	3.218	\\ \hline
7.50e-02	&	3.1485e-03	&	1.991	&	2.8936e-03	&	1.985	&		6.5336e-06	&	3.151	&	8.7048e-06	&	3.118	\\ \hline
6.00e-02	&	2.0170e-03	&	1.995	&	1.8550e-03	&	1.993	&		3.2882e-06	&	3.077	&	4.3978e-06	&	3.060	\\ \hline
4.00e-02	&	8.9745e-04	&	1.997	&	8.2582e-04	&	1.996	&		9.5810e-07	&	3.041	&	1.2860e-06	&	3.032	\\ \hline
3.00e-02	&	5.0509e-04	&	1.998	&	4.6483e-04	&	1.998	&		4.0225e-07	&	3.017	&	5.4094e-07	&	3.010	\\ \hline
\end{tabular}
\end{center}
\label{eocs}
\end{table}%

For the temporal order of convergence of the modified relaxation scheme, we take a large polynomial degree, $\ell = 5$, in order to minimize the spatial error and mesh size $h = 10^{-2}$. We then compute the temporal experimental order of convergence $\R_t$ by performing a series of  different realizations with time step lengths $k_\nu$. The results are reported in Table \ref{eoct} and confirm the second order temporal accuracy of the numerical method \eqref{FDrelax2} for both the wavefunction $u$ and the auxiliary variable $\phi$. The initialization of Section \ref{sec:initi} is used for the results of Tables \ref{eocs} and \ref{eoct}, as the naive initialization would fail to produce second order in time for $\phi.$
\begin{table}[htp]
\caption{Temporal experimental orders of convergence $\R_t$ for $u, \phi$ for scheme \eqref{FDrelax2}.}\vspace{-2mm}
\begin{center}
\begin{tabular}{||c|cc|cc||} \hline
$k_\nu$ 		& 	$E(u;h,k_\nu)$ 	& 	$\R_t$  	&  	$E(\phi;h,k_\nu)$ 	& 	$\R_t$ 	\\ \hline
2.00e-02	&	2.2440e-04	&	-		&	6.5517e-04	&	-		\\ \hline
1.00e-02	&	5.6225e-05	&	1.997	&	1.6442e-04	&	1.995	\\ \hline
5.00e-03	&	1.4078e-05	&	1.998	&	4.1174e-05	&	1.998	\\ \hline
4.00e-03	&	9.0130e-06	&	1.998	&	2.6359e-05	&	1.999	\\ \hline
2.00e-03	&	2.2549e-06	&	1.999	&	6.5940e-06	&	1.999	\\ \hline
1.00e-03	&	5.6392e-07	&	1.999	&	1.6490e-06	&	2.000	\\ \hline
\end{tabular}
\end{center}
\label{eoct}
\end{table}%

The corresponding spatial and temporal order of convergence for scheme \eqref{Delfour} are presented in Tables \ref{Deocs} and \ref{Deoct} respectively. In both cases we verify the theoretically expected rates. Furthermore, we notice that both methods produce very similar spatial and temporal errors,  with differences being observed after the $7th$-decimal digit. 
\begin{table}[htp]
\caption{Spatial experimental rates of  convergence $\R_s$ of $u$ for scheme \eqref{Delfour}.}\vspace{-2mm}
\begin{center}
\begin{tabular}{||c|cc||cc||} \hline
&  \multicolumn{2}{c||}{$\ell=1$} & \multicolumn{2}{c||}{$\ell=2$}  \\ \hline 
$h_\mu$ 	& $E(u;h_\mu,k)$ 	& 	$\R_s$  		&    $E(u;h_\mu,k)$ 	&  	$\R_s$ 		\\ \hline 
6.00e-01	&	1.7256e-01	&				&	1.3132e-02	&				\\ \hline
3.00e-01	&	4.8634e-02	&	1.827		&	7.5545e-04	&	4.119		\\ \hline
1.50e-01	&	1.2555e-02	&	1.954		&	6.8018e-05	&	3.473		\\ \hline
1.20e-01	&	8.0670e-03	&	1.982		&	3.3242e-05	&	3.209		\\ \hline
7.50e-02	&	3.1647e-03	&	1.991		&	7.6982e-06	&	3.112		\\ \hline
6.00e-02	&	2.0274e-03	&	1.996		&	3.8922e-06	&	3.056		\\ \hline
4.00e-02	&	9.0195e-04	&	1.998		&	1.1388e-06	&	3.031		\\ \hline
3.00e-02	&	5.0752e-04	&	1.999		&	4.7829e-07	&	3.016		\\ \hline
\end{tabular}
\end{center}
\label{Deocs}
\end{table}%
\begin{table}[htp]
\caption{Temporal experimental orders of convergence $\R_t$ of $u$ for  scheme \eqref{Delfour}.}\vspace{-2mm}
\begin{center}
\begin{tabular}{||c|cc||} \hline
$k_\nu$ 	& 	$E(u;h,k_\nu)$ 	& 	$\R_t$ 	\\ \hline
2.00e-02	&	6.1482e-04	&			\\ \hline
1.00e-02	&	1.5375e-04	&	1.999	\\ \hline
5.00e-03	&	3.8432e-05	&	2.000	\\ \hline
4.00e-03	&	2.4594e-05	&	2.000	\\ \hline
2.00e-03	&	6.1506e-06	&	1.999	\\ \hline
1.00e-03	&	1.5365e-06	&	2.001	\\ \hline
\end{tabular}
\end{center}
\label{Deoct}
\end{table}%
\subsection{Discrete balance of mass and energy}\label{sec:DBME}
Robust approximation of  energy influx in the non-conservative problem \eqref{WFNLS}  is a highly  desirable feature for any numerical scheme used in that context. In this section we investigate numerically the behaviour of the numerical schemes  \eqref{FDrelax2} and \eqref{Delfour} in that respect. In particular we are interesting in the long time behaviour of discrete mass and energy balance laws, and their convergence characteristics with respect to timestep $k$ for various choices of the function $r(t)$. Given that an exact solution is known, cf. equation \eqref{exact}, we proceed by computing directly the error in mass and energy. 
For a given function $r(t)$, we consider the function $u(x,t)$ given in \eqref{exact} with initial condition $u_0(x)=u(x,0),$ coefficients $p(t)=1$ and $q(t)=2\exp{\left(2\int_0^t r(s) \dif s \right) },$ and periodic boundary conditions. In   $\Omega= [-30,30]$ we take a uniform grid with mesh size $h=10^{-2}$ and we use cubic finite elements ($\ell=3$) for its discretization.  We take the final time $T=6$ and discretize $[0,T]$ with a uniform timestep $k=10^{-3}$.

We investigate a number of different scenarios:
\begin{enumerate}
\item The conservative case, $r_1(t)=0.$
\item Constant decay, $r_2(t)=1.$
\item Constant growth, $r_3(t)=-1.$
\item Growth followed by decay, $r_4=t-\frac{T}{2}.$
\item Decay followed by growth, $r_5=\sin\left(\frac{2\pi}{T}t\right).$
\item A problem that initially is conservative, then exhibits a short phase of growth, and then becomes conservative again. This is the most realistic scenario for wind induced growth of water waves: $r_6(t) = \frac{-1}{C_e} \exp\left(-\mu^2(t-\frac{T}{2})^2\right)$ where $C_e$ is a constant such that $\int_0^T r_6(t)\dif t = -1.$
\end{enumerate}
For each of these problems, first we compute the errors in mass ($\M^n_e$) and energy
 ($\E^n_e$)  for the relaxation scheme \eqref{FDrelax2},
\begin{align}
& \M^n_{e,h,k} := \big| \M_h^n - \M(t_n)\big|,  \label{masserr} \\ & \qquad\qquad\text{where}\  \M_h^n = \|U_h^n\|^2, \quad   \M(t_n) =\|u(t_n)\|^2 \notag \\
& \E^n_{e,h,k} := \big| \E_h^n - \E(t_n)\big|,   \label{enrgerr} \\ & \qquad\qquad \text{where}\  \E_h^n = \frac12 p_\nph \E_{\kappa,h}^n + \frac14 q_\nph \E_{p,h}^n , \quad \E(t_n) = \frac12 p(t_\nph) \E_\kappa(t_n) + \frac14 q(t_\nph) \E_p(t_n) \notag
\end{align}
consistently with the definitions \eqref{massenergy}, \eqref{dmbl} and \eqref{ddmebl}.

Table \ref{mberr} shows  the mass  error for these choices of  $r(t)$. For  $r_1$ we have exact conservation law, see first column,  since this case corresponds to the classical cubic NLS. The next two choices correspond to constant values with opposite signs, thus the error decays for $r_2$ see second column, while the error grows exponentially for $r_3$ as shown in the third column. Choices $r_4$ and $r_5$ have opposites signs, when one is negative the other is positive and vice versa, but the total flux is zero for both, i.e $\int_0^T r(s) \dif s = 0$. This behaviour is reflected in the errors shown in fourth and fifth columns. For  $r_4$  the error increases for $0\le t \le \frac{T}{2} $ since $r(t)$ is negative and starts decreasing for $\frac{T}{2} \le t \le T$ since $r(t)$ is positive. The exact opposite happens for $r_5$ as depicted in the fifth column, since the function is first positive and then negative. It is also interesting to notice the symmetry, around $t=3$, of the error values for both cases.  
For the last choice with $\mu=12$, $r_6$ initially vanishes, it is nonzero in a neighbourhood  around $t=\frac{T}{2}$ and vanishes afterwards. This situation is reflected in the  error,  shown in last column of the table, where initially we have conservation of mass, like the $r_1$ case.  Then, for $2.5\le t \le 3.5$ $r_6(t)$ is non-zero and we observe a jump in the error, notice the error at $t=3$, which then  stabilizes for $t\ge 4$ since $r_6(t)$ decreases exponentially. 
\begin{table}[htp]
\caption{Mass Balance  Error $(\M_{e,h,k}^n)$ for scheme \eqref{FDrelax2}.} \vspace{-2mm}
\begin{center}
\begin{tabular}{||c||c|c|c|c|c|c|} \hline 
	&	$r_1$		& 	$r_2$		& 	$r_3$		&	 $r_4$		&  	$r_5$ 		& 	$r_6$		\\  \hline   
$t=k$	&	1.3323e-15	&	5.5515e-10	&	5.5738e-10	&	6.3652e-09	&	3.7126e-13	&	1.3323e-15	\\ \hline
$t=1$	&	3.7748e-13	&	7.5283e-08	&	4.1103e-06	&	4.7384e-04	&	1.0436e-07	&	3.7748e-13	\\ \hline
$t=2$	&	8.2423e-13	&	2.0377e-08	&	6.0742e-05	&	9.2657e-03	&	3.7295e-08	&	8.2423e-13	\\ \hline
$t=3$	&	1.3127e-12	&	4.1365e-09	&	6.7324e-04	&	2.2258e-02	&	2.0301e-08	&	1.4590e-06	\\ \hline
$t=4$	&	1.5088e-12	&	7.4643e-10	&	6.6328e-03	&	9.2657e-03	&	3.7296e-08	&	1.0734e-06	\\ \hline
$t=5$	&	1.4673e-12	&	1.2627e-10	&	6.1263e-02	&	4.7384e-04	&	1.0436e-07	&	1.0734e-06	\\ \hline
$t=6$	&	1.4375e-12	&	2.0507e-11	&	5.4321e-01	&	3.0999e-11	&	3.9302e-13	&	1.0734e-06	\\ \hline
\end{tabular}
\end{center}
\label{mberr}
\end{table}%

The corresponding errors for the energy are shown in Table \ref{eberr}. The overall behaviour of the energy error is analogous to the mass error however,  there are 
some differences. In the first column the corresponding conservation of energy is up to single precision. Columns two and three show the expected decay and growth respectively, while 
for the cases $r_4$ and $r_5$ we have similar behaviour with the mass error counterparts, with the main difference being the loss of symmetry of the error around $t=3$. The errors for last case $r_6$ have the same characteristics as the corresponding mass errors, however they exhibit   some oscillatory behaviour. 
\begin{table}[htp]
\caption{Energy Balance  Error $(\E_{e,h,k}^n)$ for scheme \eqref{FDrelax2}.} \vspace{-2mm}
\begin{center}
\begin{tabular}{||c||c|c|c|c|c|c|} \hline 
	&	$r_1$		& 	$r_2$		& 	$r_3$		&	 $r_4$		&  	$r_5$ 		& 	$r_6$		\\  \hline   
$t=k$	&	1.4163e-12	&	1.1508e-09	&	1.1541e-09	&	4.6839e-09	&	2.0973e-12	&	1.4163e-12	\\ \hline
$t=1$	&	6.8242e-07	&	1.6879e-07	&	1.7301e-05	&	5.8185e-04	&	4.6087e-07	&	6.8242e-07	\\ \hline
$t=2$	&	1.1136e-06	&	5.4154e-08	&	2.2437e-04	&	1.5150e-03	&	1.4290e-07	&	1.1136e-06	\\ \hline
$t=3$	&	1.0478e-06	&	1.1519e-08	&	2.4021e-03	&	2.8081e-02	&	7.0851e-08	&	1.1792e-05	\\ \hline
$t=4$	&	7.2368e-07	&	1.2492e-09	&	2.3174e-02	&	1.5147e-02	&	9.7024e-08	&	1.9311e-06	\\ \hline
$t=5$	&	4.6416e-07	&	6.4806e-09	&	2.1025e-01	&	3.1189e-04	&	6.2012e-08	&	2.2225e-06	\\ \hline
$t=6 $	&	4.6705e-07	&	5.0604e-08	&	1.8378e+00	&	8.3199e-06	&	7.1270e-07	&	1.9023e-06	\\ \hline
\end{tabular}
\end{center}
\label{eberr}
\end{table}%
We perform now a convergence study of the mass and energy error with respect to timestep $k$ along the lines of the previous paragraph. We consider  two choices of the function $r(t)$ namely $r_5$ and $r_6$ and the exact solution \eqref{exact}, up to $T=6$. To isolate the temporal error, we consider a fine spatial discretization of $[a,b] = [-30,30]$ consisting of $M=6000$ points and quintic ($\ell=5$) spline finite element space. We then proceed to compute the convergence rate as 
\begin{equation*}
\R_m = \frac{\log\left(\M_{e,h,k_1}^n\right) - \log\left(\M_{e,h,k_2}^n\right)}{\log(k_1) -\log(k_2)}, \quad 
\R_e = \frac{\log\left(\E_{e,h,k_1}^n\right) - \log\left(\E_{e,h,k_2}^n\right)}{\log(k_1) -\log(k_2)},
\end{equation*}
for two different realizations(runs) with timesteps $k_1$ and $k_2$ respectively.  To determine numerically the convergence rates we perform a series of runs for stepsizes $k_\nu$. 
The results are reported in Table \ref{meeoc} and confirm the second order convergence rates for both mass and energy. The super-convergence shown in the first column of the table for the mass error is due to symmetry of $r_5$ in $[0, T]$ and the induced cancellations. 
\begin{table}[htp]
\caption{Mass and Energy errors convergence rates $\R_m$ and $\R_e$ for scheme \eqref{FDrelax2}. }\vspace{-2mm}
\begin{center}
\begin{tabular}{||c||cc|cc||cc|cc||} \hline
&  \multicolumn{4}{c||}{$r_5$} & \multicolumn{4}{c||}{$r_6$}  \\ \hline 
$k_\nu$ 		& $\M_{e,h,k_\nu}^n$&	$\R_m$	& 	$\E_{e,h,k_\nu}^n$ & $\R_e$	& 	$\M_{e,h,k_\nu}^n$	&$\R_m$	& 	$\E_{e,h,k_\nu}^n$ & 	$\R_e$ \\ \hline
4.00e-02	&	6.0110e-06	&			&	1.1457e-03	&			&	1.7322e-03	&			&	3.0875e-03	&			\\ \hline
2.00e-02	&	3.7501e-07	&	4.003	&	2.8512e-04	&	2.007	&	4.3029e-04	&	2.009	&	7.6359e-04	&	2.016	\\ \hline
1.00e-02	&	2.3481e-08	&	3.997	&	7.1230e-05	&	2.001	&	1.0740e-04	&	2.002	&	1.9040e-04	&	2.004	\\ \hline
5.00e-03	&	1.4699e-09	&	3.998	&	1.7811e-05	&	2.000	&	2.6839e-05	&	2.001	&	4.7568e-05	&	2.001	\\ \hline
2.50e-03	&	9.0880e-11	&	4.016	&	4.4537e-06	&	2.000	&	6.7092e-06	&	2.000	&	1.1890e-05	&	2.000	\\ \hline
1.25e-03	&	5.6529e-12	&	4.007	&	1.1136e-06	&	2.000	&	1.6773e-06	&	2.000	&	2.9723e-06	&	2.000	\\ \hline
\end{tabular}
\end{center}
\label{meeoc}
\end{table}%
The series of results in this section verify the effectiveness and robustness of the relaxation scheme \eqref{FDrelax2}.  

To investigate mass and energy balance for the DFP scheme \eqref{Delfour}, we use the same setup as for the relaxation scheme, the only difference is that $\|U^{n}_h\|_{L^4}^4$ is used instead of $\mathcal{E}^n_{p,h}.$
 Tables \ref{Nmberr} and \ref{Neberr} should be compared  to Tables \ref{mberr} and \ref{eberr}  for the mass and energy balance error  $\M_{e,h,k}^n$, $\E_{e,h,k}^n$ respectively, for scheme \eqref{Delfour}.

\begin{table}[htp]
\caption{Mass  Balance Error $(\M_{e,h,k}^n)$ for scheme \eqref{Delfour}.} \vspace{-2mm}
\begin{center}
\begin{tabular}{||c||c|c|c|c|c|c||} \hline 
		&	$r_1$		& 	$r_2$		& 	$r_3$		&	 $r_4$		&  	$r_5$ 		& 	$r_6$		\\  \hline   
$t=k$	&	2.8866e-15	&	1.9647e-10	&	1.9687e-10	&	2.2437e-09	&	1.3012e-13	&	2.4425e-15	\\ \hline
$t=1$	&	3.1286e-13	&	7.2351e-08	&	5.3461e-07	&	1.3751e-05	&	5.9476e-08	&	3.1286e-13	\\ \hline
$t=2$	&	6.2639e-13	&	5.3233e-08	&	2.9064e-06	&	6.0000e-05	&	5.5232e-08	&	6.2639e-13	\\ \hline
$t=3$	&	9.4236e-13	&	2.9375e-08	&	1.1851e-05	&	8.7418e-05	&	4.8462e-08	&	8.5044e-07	\\ \hline
$t=4$	&	1.2597e-12	&	1.4409e-08	&	4.2951e-05	&	6.0000e-05	&	5.5232e-08	&	1.0316e-06	\\ \hline
$t=5$	&	1.5712e-12	&	6.6258e-09	&	1.4594e-04	&	1.3751e-05	&	5.9475e-08	&	1.0316e-06	\\ \hline
$t=6$	&	1.8834e-12	&	2.9250e-09	&	4.7605e-04	&	8.3311e-13	&	6.6414e-13	&	1.0316e-06	\\ \hline
\end{tabular}
\end{center}
\label{Nmberr}
\end{table}%
\begin{table}[htp]
\caption{Energy Balance  Error $(\E_{e,h,k}^n)$ for scheme \eqref{Delfour}.} \vspace{-2mm}
\begin{center}
\begin{tabular}{||c||c|c|c|c|c|c||} \hline 
		&	$r_1$		& 	$r_2$		& 	$r_3$		&	 $r_4$		&  	$r_5$ 		& 	$r_6$		\\  \hline   
$t=k$	&	5.5511e-16	&	4.8423e-10	&	4.8618e-10	&	1.1099e-09	&	2.7989e-13	&	9.9920e-16	\\ \hline
$t=1$	&	2.5180e-13	&	6.5665e-08	&	3.5852e-06	&	3.7644e-05	&	8.9841e-08	&	2.5202e-13	\\ \hline
$t=2$	&	4.9150e-13	&	1.7774e-08	&	5.2983e-05	&	6.9807e-04	&	3.7014e-08	&	4.9072e-13	\\ \hline
$t=3$	&	7.7116e-13	&	3.6081e-09	&	5.8724e-04	&	4.0795e-03	&	1.9365e-08	&	4.1036e-07	\\ \hline
$t=4$	&	1.0706e-12	&	6.5108e-10	&	5.7855e-03	&	6.9807e-04	&	3.7014e-08	&	3.0192e-07	\\ \hline
$t=5$	&	1.3782e-12	&	1.1014e-10	&	5.3437e-02	&	3.7643e-05	&	8.9841e-08	&	3.0192e-07	\\ \hline
$t=6$	&	1.6597e-12	&	1.7887e-11	&	4.7381e-01	&	5.2053e-12	&	1.2291e-12	&	3.0192e-07	\\ \hline
\end{tabular}
\end{center}
\label{Neberr}
\end{table}%

\section{Discussion}\label{sec:outlook}

\subsection{Mass balance error comparison}
Both schemes conserve the mass at the discrete level in the case of constant coefficients  $p(t)=p_0, \ q(t)=q_0$ and $r(t)\equiv 0$ which corresponds to the $r_1$-column of  Tables \ref{mberr} and \ref{Nmberr}. In the non-conservatives cases, $r_2, \ r_3$ and $r_4$,  the mass balance error  exhibits similar behaviour for both schemes, however the corresponding error values in these cases are 1-2 orders of magnitude smaller for scheme \eqref{Delfour} compared to the relaxation scheme  \eqref{FDrelax2}. Furthermore,  for $r_5$ and $r_6$ the two schemes have almost identical behaviour as it's shown in the corresponding columns of Tables  \ref{mberr} and \ref{Nmberr} respectively. 

\subsection{Energy balance error comparison}
Scheme \eqref{Delfour}, by design, conserves  energy at the discrete level. This is depicted in the  1st column of Table \ref{Neberr} where the energy is conserved to double precision, while the relaxation scheme \eqref{FDrelax2} only up to single precision as depicted in the corresponding column of Table  \ref{eberr}. This is the main difference between the schemes. For the other non-conservative choices $r_2, \ r_3$ and $r_4$ the values in the corresponding columns of Tables \ref{eberr} and \ref{Neberr} are comparable,  with those for scheme \eqref{Delfour} being slightly smaller. Lastly the values of energy balance error for cases $r_5$ and $r_6$ are very similar between the two methods. 

\subsection{Computational comparison}
Overall both schemes behave equally well when preserving the mass, while scheme \eqref{Delfour} has a clear advantage over the relaxation scheme  \eqref{FDrelax2} when conservation of energy is concerned. However the computational cost of scheme \eqref{Delfour}  is significantly higher than of the relaxation one since an elaborate Newton-like scheme is used for the solution of the nonlinear system, see \cite{ADK} for details. On the other hand the relaxation scheme introduces an auxiliary variable thus doubling the number of equations need it to be solved. Using the same discretization and physical parameters for both schemes, we observe that,  on the average,  to produce any of the columns of Tables \ref{mberr} and \ref{eberr} it requires about $106s$ of computational time, while the corresponding time for scheme \eqref{Delfour} is about $175s$, a $65\%$ increase. 

\subsection{Outlook}
The relaxation scheme \eqref{FDrelax2} is simpler to implement and analyze, by requiring the solution of only a complex  linear system at every timestep. Crucially, this allows practical a posteriori error estimators to be constructed, see \cite{KK} for the conservative case. Thus, it is extremely promising for problems where adaptivity is required.
The implicit scheme \eqref{Delfour}, implemented with the Newton-type scheme introduced in \cite{ADK}, is only moderately more computationally expensive. Moreover, building on the ideas behind it, even higher order conservative in time and computational efficient schemes can be proposed.
%
%
%

\end{document}